\documentclass[graybox]{svmult}

\usepackage{amscd,amsfonts,amssymb,amsmath}

\usepackage{mathptmx}       
\usepackage{helvet}         
\usepackage{courier}        
\usepackage{type1cm}        
\usepackage{makeidx}         
\usepackage{graphicx}        
\usepackage{multicol}        
\usepackage[bottom]{footmisc}

\makeindex             

\begin{document}

\title*{Braid groups  in handlebodies and corresponding Hecke algebras}
\author{Valeriy G. Bardakov}

\institute{Valeriy G. Bardakov \at Sobolev Institute of Mathematics, Novosibirsk State University, Novosibirsk State Agrarian University, Novosibirsk 630090, Russia, \email{bardakov@math.nsc.ru}}

%
%
\maketitle


\abstract{In this paper we study the kernel of the homomorphism $B_{g,n} \to B_n$ of the braid group $B_{g,n}$ in the handlebody $\mathcal{H}_g$ to the braid group $B_n$. We prove that this kernel is  semi-direct product of free groups. Also, we introduce an algebra $H_{g,n}(q)$, which is some analog of the Hecke algebra $H_n(q)$, constructed by the braid group~$B_n$.}

\section{Introduction}

Let $\mathcal{H}_g$ be a handlebody of genus $g$. The braid group  $B_{g,n}$ on $n$ stings in the handlebody $\mathcal{H}_g$ was introduced by A.~B.~Sossinsky \cite{S} and independently S.~Lambropoulou \cite{L1}. Properties of this group are studied by V.~V.~Vershinin \cite{V, V1} and by S.~Lambropolou \cite{L1, L2}.

The motivation for studying braids from $B_{g,n}$ comes from studying oriented knots and links in knot complements in compact connected oriented 3-manifolds and in handlebodies, since these spaces may be represented by a fixed braid or a fixed integer-framed braid in  $\mathbb{S}^3$ \cite{LR1, L1, HL, LR2}. Then knots and links in these spaces may be represented by elements of the braid group $B_{g,n}$ or an appropriate cosets of these group \cite{L1}. In particular, if $M$ denotes the complements of the $g$-unlink or a connected sum of $g$ lens spaces of type $L(p, 1)$ or a handlebody of genus $g$, then knots and links in these spaces may be represented precisely by the braids in $B_{g,n}$ for $n \in \mathbb{N}$. In the case $g=1$, $B_{1,n}$ is the Artin group of type  $\mathcal{B}$.

 The group $B_{g,n}$ can be considered as a subgroup of the  braid group $B_{g+n}$ on $g+n$ strings such that the braids in $B_{g,n}$ leave the first $g$ strings identically fixed. Using this fact V. V. Vershinin \cite{V} and S.~Lambropoulou \cite{L1} defined a epimorphism of $B_{g,n}$ onto the symmetric group $S_n$ and prove that the kernel of this epimorphism: $P_{g,n}$ is a subgroup of the pure braid group $P_{g+n}$ and is semi-direct products of free groups.
On the other side, V.~V.~Vershinin \cite{V} noted that there is a decomposition $B_{g,n} = R_{g,n} \leftthreetimes B_n$ for some group $R_{g,n}$ and the braid group $B_n$. S.~Lambropoulou proved in \cite{L2} that $R_{1,n}$ is isomorphic to a free group $F_n$ of rank $n$.

The Hecke algebra of type $\mathrm{A}$, $H_n(q)$ was used by V.~F.~R.~Jones \cite{J} for the construction of polynomial invariant for classical links, the well known HOMFLYPT polynomial. S.~Lambropoulou \cite{L2} defined a generalization of the Hecke algebra of type $\mathrm{B}$ to construct a HOMFLYPT-type polynomial invariant for links in the solid torus, which was then used in \cite{DiLaPr} for extending the study to the lens spaces $L(p,1)$.

In the present paper we define some decomposition of $P_{g,n}$ as  semi-direct product of free groups, which is different from the decomposition defined in \cite{L1, V}.
Also, we study the group $R_{g,n}$  and prove that this group is  semi-direct products of  free groups.
Using this decomposition  and the decomposition $B_{g,n} = R_{g,n} \leftthreetimes B_n$ we  define some algebra, which is a generalization of the Hecke algebra $H_n(q)$.

The paper is organized as follows. In Section 2, we remind some facts on the braid group $B_n$. In particular, we define vertical and horizontal decompositions of the pure braid group $P_n$. In Section 3, we recall some facts on the group $B_{g,n}$, we describe the vertical decomposition of $P_{g,n}$ and we define the horizontal decomposition of $P_{g,n}$. In Section 4, we study  $R_{g,n}$   and we construct vertical and horizontal decompositions for this group. In Section 5 we introduce some algebra $H_{g,n}(q)$ which is a generalization of the Hecke  algebra $H_{n}(q)$ and we find the quotient of $B_{g,n}$ by the relations $\sigma_i^2 = 1$.

\subsection*{Acknowledgements. } The author gratefully acknowledges Prof. Lambro\-pou\-lou and her students and colleagues: Neslihan G\"{u}g\"{u}mc\"{u}, Stathis Antoniou, Dimos Goundaroulis, Ioannis Diamantis, Dimitrios Kodokostas for the kind invitation to the Athens, where this paper was written, for  conversation and interesting discussions.

This work was supported by the  Russian Foundation for Basic Research (project 16-01-00414).

\section{Braid group}

In this section we recall some facts on the braid groups (see \cite{Bir, M}).

The braid group $B_m$, $m \geq 2$, on $m$ strings is generated by elements
$$
\sigma_1, \sigma_2, \ldots, \sigma_{m-1},
$$
and is defined by relations
$$
\begin{array}{ll}
\sigma_i \sigma_j = \sigma_j \sigma_i, & ~\mbox{for}~|i - j| >1, \\
\sigma_i \sigma_{i+1} \sigma_i = \sigma_{i+1} \sigma_i \sigma_{i+1}, & ~\mbox{for}~i = 1, 2, \ldots, m-2.
\end{array}
$$
A subgroup of $B_m$ which is generated by elements
$$
a_{ij} = \sigma_{j-1} \sigma_{j-2} \ldots \sigma_{i+1} \sigma_{i}^2 \sigma_{i+1}^{-1} \ldots \sigma_{j-2}^{-1} \sigma_{j-1}^{-1},~~~1 \leq i < j \leq m,
$$
is called the {\it pure braid group} and is denoted $P_m$. This group is defined by the relations
 \begin{align}
& a_{ik} a_{ij} a_{kj} = a_{kj} a_{ik} a_{ij},   \label{re2}\\
& a_{nj} a_{kn} a_{kj} = a_{kj} a_{nj} a_{kn},  ~\mbox{for}~n < j, \label{re3}\\
& (a_{kn} a_{kj} a_{kn}^{-1}) a_{in} = a_{in} (a_{kn} a_{kj} a_{kn}^{-1}),  ~\mbox{for}~i < k < n < j, \label{re4}\\
& a_{kj} a_{in} = a_{in} a_{kj},  ~\mbox{for}~k < i < n < j ~\mbox{or}~n < k. \label{re1}
\end{align}
The subgroup $P_m$ is normal in $B_m$, and the quotient $B_m / P_m$ is the symmetric group $S_m$. The generators of $B_m$ act on the generator $a_{ij} \in P_m$ by the rules:
 \begin{align}
& \sigma_k^{-1} a_{ij} \sigma_k =  a_{ij},  ~\mbox{for}~k \not= i-1, i, j-1, j, \label{c1}\\
& \sigma_{i}^{-1} a_{i,i+1} \sigma_{i} =  a_{i,i+1},   \label{c2}\\
& \sigma_{i-1}^{-1} a_{ij} \sigma_{i-1} =   a_{i-1,j},   \label{c3}\\
& \sigma_{i}^{-1} a_{ij} \sigma_{i} =  a_{i+1,j} [a_{i,i+1}^{-1}, a_{ij}^{-1}],  ~\mbox{for}~j \not= i+1 \label{c4}\\
& \sigma_{j-1}^{-1} a_{ij} \sigma_{j-1} =  a_{i,j-1},   \label{c5}\\
& \sigma_{j}^{-1} a_{ij} \sigma_{j} =  a_{ij} a_{i,j+1} a_{ij}^{-1},   \label{c6}
\end{align}
where $[a, b] = a^{-1} b^{-1} a b$.

Denote by
$$
U_{i} = \langle a_{1i}, a_{2i}, \ldots, a_{i-1,i} \rangle,~~~i = 2, \ldots, m,
$$
a subgroup of $P_m$.
It is known that $U_i$ is a free group of rank $i-1$. One can rewrite the relations of $P_m$ as the following conjugation rules (for $\varepsilon = \pm 1$):
  \begin{align}
& a_{ik}^{-\varepsilon} a_{kj}  a_{ik}^{\varepsilon} = (a_{ij} a_{kj})^{\varepsilon} a_{kj} (a_{ij} a_{kj})^{-\varepsilon},  \label{co1}\\
& a_{kn}^{-\varepsilon} a_{kj}  a_{kn}^{\varepsilon} = (a_{kj} a_{nj})^{\varepsilon} a_{kj} (a_{kj} a_{nj})^{-\varepsilon},  ~\mbox{for}~n < j, \label{co2}\\
& a_{in}^{-\varepsilon} a_{kj}  a_{in}^{\varepsilon} = [a_{ij}^{-\varepsilon}, a_{nj}^{-\varepsilon}]^{\varepsilon} a_{kj} [a_{ij}^{-\varepsilon}, a_{nj}^{-\varepsilon}]^{-\varepsilon},  ~\mbox{for}~i < k < n, \label{co3}\\
& a_{in}^{-\varepsilon} a_{kj} a_{in}^{\varepsilon} = a_{kj},  ~\mbox{for}~k < i < n < j ~\mbox{or}~  n < k. \label{co4}
\end{align}

From these rules it follows that $U_m$
 is normal in $P_m$ and hence $P_m$  has the following decomposition: $P_m = U_m \leftthreetimes P_{m-1}$, where the action of $P_{m-1}$ on $U_m$ is define by the rules (\ref{co1}-\ref{co4}). By induction on $m$, $P_m$ is the semi-direct product of free groups:
$$
P_m =U_m \leftthreetimes (U_{m-1} \leftthreetimes (\ldots  \leftthreetimes (U_3 \leftthreetimes U_2)\ldots )).
$$
We will call this decomposition  {\it vertical decomposition}.

Let $U_m^{(k)}$, $k = 1, 2, \ldots, m$, be the subgroup of $P_m$ which is generated by $a_{ij}$, where $k < j$. Then $U_m^{(k)} = U_m \leftthreetimes (U_{m-1} \leftthreetimes (\ldots  \leftthreetimes (U_{k+2}  \leftthreetimes U_{k+1})\ldots ))$. By  definition $U_m^{(1)} = P_m$ and these groups form the normal series
$$
1 = U_m^{(m)} \leq U_m^{(m-1)}  \leq \ldots \leq U_m^{(2)}  \leq U_m^{(1)} = P_m,
$$
where
$$
U_m^{(r)} / U_m^{(r+1)} \cong F_{r}, ~~~r = 1, 2, \ldots, m-1.
$$

On the other side,  define the following subgroups of $P_m$:
$$
V_k = \langle a_{k,k+1}, a_{k,k+2}, \ldots, a_{k,m} \rangle,~~k = 1, 2, \ldots, m-1.
$$
This group is free of rank $m-k$.  Using the defining relation of $P_m$, it is not difficult to prove the following:

\begin{lemma}\label{l1}
In $P_n$ hold the following conjugation rules (for $\varepsilon = \pm 1$):

1) $a_{kj}^{-\varepsilon} a_{ik} a_{kj}^{\varepsilon} = (a_{ik} a_{ij})^{\varepsilon} a_{ik} (a_{ik} a_{ij})^{-\varepsilon},$ where $i < k < j$;

2) $a_{jk}^{-\varepsilon} a_{ik} a_{jk}^{\varepsilon} = (a_{ij} a_{ik})^{\varepsilon} a_{ik} (a_{ij} a_{ik})^{-\varepsilon},$ where $i < j < k$;

3) $a_{kn}^{-\varepsilon} a_{ij} a_{kn}^{\varepsilon} = [a_{ik}^{-\varepsilon},  a_{in}^{-\varepsilon}]^{\varepsilon} a_{ij} [a_{ik}^{-\varepsilon},  a_{in}^{-\varepsilon}]^{-\varepsilon},$ where $i < k < j < n$;

4) $a_{in}^{-\varepsilon} a_{kj} a_{in}^{\varepsilon} = a_{kj},$ where $k < i$ and $n < j$;

5) $a_{kj}^{-\varepsilon} a_{in} a_{kj}^{\varepsilon} = a_{in},$ where $n < k$.
\end{lemma}

From this lemma it follows that $V_1$ is normal in $P_m$ and we have decomposition $P_m = V_1 \leftthreetimes P_{m-1}$. By induction on $m$, $P_m$ is the semi-direct products of free groups:
$$
P_m =V_1 \leftthreetimes (V_{2} \leftthreetimes (\ldots  \leftthreetimes (V_{m-2}  \leftthreetimes V_{m-1})\ldots )).
$$
We will call this decomposition  {\it horizontal decomposition}. Let $V_m^{(k)}$ be a subgroup of $P_m$ which is generated by $a_{ij}$ for $i < k$. Then
we have the normal series
$$
1 = V_m^{(1)} \leq V_m^{(2)}  \leq \ldots \leq V_m^{(m-1)}  \leq V_m^{(m)} = P_m,
$$
where
$$
V_m^{(r)} / V_m^{(r-1)} \cong F_{m-r+1}, ~~~r =  2, 3, \ldots, m.
$$

A motivation for the  terms vertical and horizontal is as follows. If we put the generators of $P_m$ in the following table
$$
\begin{array}{ccccc}
  a_{12}, &  a_{13}, & \ldots &  a_{1,m-1}, & a_{1,m}, \\
 &  a_{23}, & \ldots & a_{2,m-1}, & a_{2,m}, \\
\ldots & \ldots & \ldots & \ldots & \ldots \\
 &  &  & a_{m-2,m-1}, & a_{m-2,m}, \\
 & &  &  & a_{m-1,m},
\end{array}
$$
then the generators from the $k$th row generate $V_k$ and the generators from the $k$th column  generate $U_{k+1}$. The group $U_m^{(r)}$ is generated by the last $m-r$ columns of this table and the group $V_m^{(r)}$ is generated by the first $r-1$ rows of this table.

\section{Braid groups in handlebodies}
 Recall some facts on the braid group  $B_{g,n}$ on $n$ strings in the handlebody $\mathcal{H}_g$ (see \cite{L1, S, V}).
The group
$B_{g,n}$ is  generated by elements
$$
\tau_1, \tau_2, \ldots \tau_g, \sigma_{g+1}, \sigma_{g+2}, \ldots, \sigma_{g+n-1},
$$
and is defined by the  following list of relations
$$
\begin{array}{ll}
\sigma_i \sigma_j = \sigma_j \sigma_i, & ~\mbox{for}~|i - j| >1, \\
\sigma_i \sigma_{i+1} \sigma_i = \sigma_{i+1} \sigma_i \sigma_{i+1}, & ~\mbox{for}~i = g+1, \ldots, g+n-2, \\
\tau_k \sigma_i = \sigma_i \tau_k, & ~\mbox{for}~k \geq 1, ~~i \geq g+2, \\
\tau_k (\sigma_{g+1} \tau_k \sigma_{g+1}) = (\sigma_{g+1} \tau_k \sigma_{g+1}) \tau_{k}, & ~\mbox{for}~k = 1, \ldots, g, \\
\tau_k (\sigma_{g+1}^{-1} \tau_{k+l} \sigma_{g+1}) = (\sigma_{g+1}^{-1} \tau_{k+l} \sigma_{g+1}) \tau_{k}, & ~\mbox{for}~k = 1, \ldots, g-1, ~~l = 1,  \ldots, g-k.
\end{array}
$$

The group $B_{g,n}$ can be considered as a subgroup of the classical braid group $B_{g+n}$ on $n+g$ strings such that the braids from $B_{g,n}$ leave the first $g$ strings unbraided. Then $\tau_k = a_{k,g+1}$, in $B_{g+n}$, i. e.
$$
\tau_k = \sigma_g \sigma_{g-1} \ldots \sigma_{k+1} \sigma_{k}^2 \sigma_{k+1}^{-1} \ldots \sigma_{g-1}^{-1} \sigma_{g}^{-1},~~~k = 1, 2, \ldots, g.
$$
The elements $\tau_k$, $k = 1, 2, \ldots, g$, generate a free group  of rank $g$ which is isomorphic to $U_{g+1} = \langle a_{1,g+1}, a_{2,g+1}, \ldots, a_{g,g+1} \rangle$ in $B_{g+n}$. Also, we see that some other generators of $P_{g+n}$ lie in $B_{g,n}$. Put them in the table bellow:
$$
\begin{array}{ccccc}
  a_{1,g+1}, &  a_{1,g+2}, & \ldots &  a_{1,g+n-1}, & a_{1,g+n}, \\
  a_{2,g+1}, &  a_{2,g+2}, & \ldots &  a_{2,g+n-1}, & a_{2,g+n}, \\
\ldots & \ldots & \ldots & \ldots & \ldots \\
  a_{g-1,g+1}, &  a_{g-1,g+2}, & \ldots &  a_{g-1,g+n-1}, & a_{g-1,g+n}, \\
  a_{g,g+1}, &  a_{g,g+2}, & \ldots &  a_{g,g+n-1}, & a_{g,g+n}, \\
  \hline
 &  a_{g+1,g+2}, & \ldots &  a_{g+1,g+n-1}, & a_{g+1,g+n}, \\
 & \ldots & \ldots & \ldots & \ldots \\
  &   &  &  a_{g+n-2,g+n-1}, & a_{g+n-2,g+n}, \\
   &   &  &   & a_{g+n-1,g+n}. \\
\end{array}
$$

We will denote by $\widetilde{B}_n$ the subgroup of $B_{g,n}$ which is generated by $\sigma_{g+1},$ $\sigma_{g+2},$ $\ldots, \sigma_{g+n-1}$. It is evident that
$\widetilde{B}_n$ is isomorphic to $B_n$. The corresponding pure braid group of $\widetilde{B}_n$ will be denote $\widetilde{P}_n$. This group is isomorphic to $P_n$ and is generated by elements from the above table which lie under the horizontal line. The group $\widetilde{P}_n$ has the following vertical decomposition
$$
\widetilde{P}_n =\widetilde{U}_n \leftthreetimes (\widetilde{U}_{n-1} \leftthreetimes (\ldots  \leftthreetimes (\widetilde{U}_3 \leftthreetimes \widetilde{U}_2)\ldots )),
$$
where
$$
\widetilde{U}_i = \langle  a_{g+1,g+i},  a_{g+2,g+i},  \ldots,  a_{g+i-1,g+i}\rangle,~~~i = 2, 3, \ldots,n.
$$

There is an epimorphism:
$$
\psi_n : B_{g,n} \longrightarrow  S_n,
$$
which is defined by the rule
$$
\psi_n(\tau_k) = 1,~k = 1, 2, \ldots, g,~~~\psi_n(\sigma_i) = (i, i+1),~i = g+1, g+2, \ldots, g+n-1.
$$
This epimorphism is induced by the standard epimorphism $B_{g+n} \longrightarrow  S_{g+n}$. Let $P_{g,n} = \mathrm{Ker} (\psi_n)$. Then $P_{g,n}$ is generated by the element from the table above.
In \cite{L1, V}  it was proved that there exists the following short exact sequence:
$$
1 \longrightarrow P_{g,n} \longrightarrow B_{g,n} \longrightarrow S_n \longrightarrow 1,
$$
and was found the vertical decomposition of $P_{g,n}$:
$$
P_{g,n} = U_{g+n} \leftthreetimes (U_{g+n-1} \leftthreetimes ( \ldots  \leftthreetimes (U_{g+2}  \leftthreetimes U_{g+1})\ldots )).
$$
If we let $U_{g+n}^{(g+n-i)} = U_{g+n} \leftthreetimes (U_{g+n-1} \leftthreetimes ( \ldots  \leftthreetimes (U_{g+n-i+2}  \leftthreetimes U_{g+n-i+1})\ldots ))$,
then we get the normal series:
$$
1 = U_{g+n}^{(g+n)} \leq U_{g+n}^{(g+n-1)}  \leq  \ldots \leq U_{g+n}^{(g)}   = P_{g,n},
$$
where
$$
U_{g+n}^{(r)} / U_{g+n}^{(r+1)} \cong F_r, ~~~r = g, g+1, \ldots, g+n-1.
$$
The homomorphism which is induced by the embedding $B_{g,n} \longrightarrow B_{g+n}$ sends this normal series to the corresponding normal series for $P_{g+n}$.

Let us construct the horizontal  decomposition  for $P_{g,n}$. To do this, define the following subgroups in $P_{g,n}$:
$$
\begin{array}{l}
V_{g,1} = \langle a_{1,g+1}, a_{1,g+2}, \ldots, a_{1,g+n} \rangle,  \\
V_{g,2} = \langle a_{2,g+1}, a_{2,g+2}, \ldots, a_{2,g+n} \rangle, \\
.......................................................\\
V_{g,g} = \langle a_{g,g+1}, a_{g,g+2}, \ldots, a_{g,g+n} \rangle, \\
V_{g,g+1} = \langle a_{g+1,g+2}, a_{g+1,g+3}, \ldots, a_{g+1,g+n} \rangle,  \\
V_{g,g+2} = \langle a_{g+2,g+3}, a_{g+2,g+4}, \ldots, a_{g+2,g+n} \rangle, \\
.......................................................\\
V_{g,g+n-1} = \langle a_{g+n-1,g+n} \rangle. \\
\end{array}
$$
We see that $V_{g,g+i} = V_{g+i}$ for all $i = 1, 2, \ldots, n-1$,  these subgroups
lie in $\widetilde{P}_n$ and as we know
$$
\widetilde{P}_n = V_{g,g+1} \leftthreetimes (V_{g,g+2} \leftthreetimes ( \ldots  \leftthreetimes (V_{g,g+n-2}  \leftthreetimes V_{g,g+n-1}) \ldots ))
$$
is the horizontal decomposition of $\widetilde{P}_n$.

We see that the vertical decomposition of $P_{g,n}$ is a part of the vertical decomposition for $P_{g+n}$. For the horizontal decomposition situation is more complicated.

\begin{example}
The horizontal decomposition of $P_4$ has the form
$$
P_4 = V_1 \leftthreetimes (V_2 \leftthreetimes V_3),
$$
where
$$
V_1 = \langle a_{12}, a_{13}, a_{14} \rangle,~~V_2 = \langle a_{23}, a_{24} \rangle,~~V_3 = \langle  a_{34} \rangle,
$$
and $V_1$ is normal in $P_4$, $V_2$ is normal in $V_2 \leftthreetimes V_3$.
The group $P_{2,2}$ contains subgroups
$$
V_{2,1} = \langle a_{13}, a_{14} \rangle,~~V_{2,2} = V_2 = \langle a_{23}, a_{24} \rangle,~~V_{2,3} = V_3 = \langle  a_{34} \rangle,
$$
but in this case $V_{2,1}$ is not normal in $P_{2,2}$, Indeed, from Lemma \ref{l1} we have the following relations
$$
a_{23}^{-\varepsilon} a_{13} a_{23}^{\varepsilon} = (a_{12} a_{13})^{\varepsilon} a_{13} (a_{12} a_{13})^{-\varepsilon},~~
a_{24}^{-\varepsilon} a_{14} a_{24}^{\varepsilon} = (a_{12} a_{14})^{\varepsilon} a_{14} (a_{12} a_{14})^{-\varepsilon}.
$$
Hence, $P_{2,2}$ contains not only $V_{2,1}$ but its normal closure in $P_{2,2}$ and we get the horizontal decomposition
$$
P_{2,2} = \overline{V}_{2,1} \leftthreetimes (V_2 \leftthreetimes V_3),
$$
where $\overline{V}_{2,1} = \langle V_{2,1} \rangle^{P_{2,2}}$ is the normal closure of $V_{2,1}$ in $P_{2,2}$.
\end{example}

\medskip

In the general case, let $\overline{V}_{g,i}$
be the normal closure of $V_{g,i}$ in the subgroup $\langle V_{g,i}, V_{g,i+1}, \ldots, V_{g,g}, \widetilde{P}_{n} \rangle$, i.e.
$$
\overline{V}_{g,i} = \langle V_{g,i} \rangle^{\langle V_{g,i}, V_{g,i+1}, \ldots, V_{g,g}, \widetilde{P}_{n} \rangle}.
$$

\begin{lemma}\label{l2}
1) $\overline{V}_{g,g} = V_{g,g} \cong F_n.$

2) $\overline{V}_{g,i}$ is a subgroup of $V_i$ for every $i = 1, 2, \ldots, g$ and, in particular, is a free group.
\end{lemma}

\begin{proof}
1) We see that $V_{g,g} = V_g$ and from Lemma \ref{l1} it follows that $V_g$ is normal in $\langle V_{g,g}, \widetilde{P}_{n} \rangle$.

2) The fact that $\overline{V}_{g,i}$ is a subgroup of $V_i$ follows from the conjugation rules of Lemma \ref{l1}. The fact that $\overline{V}_{g,i}$ is free follows from the fact that $V_i$ is free.
\end{proof}

Since $\widetilde{P}_n \cong P_n$, it has the horizontal decomposition:
$$
\widetilde{P}_n = V_{g+1} \leftthreetimes (V_{g+2} \leftthreetimes ( \ldots  \leftthreetimes (V_{g+n-2} \leftthreetimes V_{g+n-1}) \ldots )).
$$

Using Lemma \ref{l2}, one can construct the horizontal decomposition of $P_{g,n}$.

\begin{theorem} The group $P_{g,n}$ is the semi-direct products of  groups:
$$
P_{g,n} = \overline{V}_{g,1} \leftthreetimes (\overline{V}_{g,2} \leftthreetimes ( \ldots  \leftthreetimes (\overline{V}_{g,g}  \leftthreetimes \widetilde{P}_n) \ldots )).
$$
\end{theorem}

\begin{proof}
Use induction on $g$. If $g=1$, then $P_{1,n} = P_{n+1}$ and the horizontal decomposition for $P_{n+1}$ gives the horizontal decomposition:  $P_{1,n} = V_{g} \leftthreetimes \widetilde{P}_n$. As follows from Lemma \ref{l2}, $V_g = \overline{V}_{g,g}$.

Let $g > 1$. Define a homomorphism of $P_{g,n}$ onto the group $P_{g-1,n}$ which sends all generators of $V_{g,1}$ to the unit and keeps all other generators. The kernel of this homomorphism  is the normal closure of $V_{g,1}$ in $P_{g,n}$. Denote this kernel by $\overline{V}_{g,1} = \langle V_{g,1} \rangle^{P_{g,n}}$. Since, $V_{g,1}$ is a subgroup in $V_1$ and $V_1$ is normal in $P_{g+n}$ we get that $\overline{V}_{g,1}$ lies in $V_1$ and hence is a free. We have the decomposition $P_{g,n} = \overline{V}_{g,1} \leftthreetimes P_{g-1,n}$. Using the induction hypothesis we get the required decomposition.
\end{proof}

\section{The kernel of the epimorphism $B_{g,n} \to \widetilde{B}_n$}

As was noted in \cite{V} there is an epimorphism
$$
\varphi_n : B_{g,n} \longrightarrow  \widetilde{B}_n,
$$
where the subgroup $\widetilde{B}_n = \langle \sigma_{g+1}, \sigma_{g+2}, \ldots, \sigma_{g+n-1} \rangle$  is isomorphic to $B_n$.
This endomorphism is defined by the rule
$$
\varphi_n(\tau_k) = 1,~k = 1, 2, \ldots, g,~~~\varphi_n(\sigma_i) = \sigma_i,~i = g+1, g+2, \ldots, g+n-1.
$$
If we denote $R_{g,n} = \mathrm{Ker}(\varphi_n)$, then $B_{g,n} = R_{g,n} \leftthreetimes \widetilde{B}_n$.
 The  purpose of this section is the description of the group $R_{g,n}$.

Considering the table  with the generators of $P_{g,n}$ (see Section 3), we see that all generators, which lie in the fist $g$ rows  of this table, are elements of $R_{g,n}$.
Denote by $Q_{g,n}$ the subgroup of $R_{g,n}$ that is generated by these elements, i.e.
 $Q_{g,n} = \langle V_{g,1}, V_{g,2}, \ldots, V_{g,g} \rangle$. Then
$R_{g,n} = \langle Q_{g,n} \rangle^{B_{g,n}}$
is the normal closure of $Q_{g,n}$ in $B_{g,n}$.

On the other side, if we denote
$$
U_{g,g+i} = \langle a_{1,g+i}, a_{2,g+i}, \ldots, a_{g,g+i} \rangle \leq U_{g+i},~~~i = 1, 2, \ldots, n,
$$
then $Q_{g,n} = \langle U_{g,g+1}, U_{g,g+2}, \ldots, U_{g,g+n} \rangle$. Note that $U_{g,g+1} = U_{g+1}$.
Let $\overline{U}_{g,g+i}$ be the normal closure of $U_{g,g+i}$ in $U_{g+i}$, $i = 1, 2, \ldots, n$.  In this notations it holds:

\begin{theorem}
The group $R_{g,n}$ has the following decompositions:

\medskip

1)~ $R_{g,n} = \overline{U}_{g,g+n} \leftthreetimes (\overline{U}_{g,g+n-1} \leftthreetimes ( \ldots \leftthreetimes (\overline{U}_{g,g+2} \leftthreetimes \overline{U}_{g,g+1})\ldots )),$\\
where $\overline{U}_{g,g+1} = U_{g+1}$.

\medskip

2)~ $R_{g,n} = \overline{V}_{g,1} \leftthreetimes (\overline{V}_{g,2} \leftthreetimes ( \ldots \leftthreetimes (\overline{V}_{g,g-1} \leftthreetimes \overline{V}_{g,g})\ldots )),$\\
where $\overline{V}_{g,g} = V_g$.
\end{theorem}

\begin{proof}
As we know, $R_{g,n}$ is the normal closure of $Q_{g,n}$ in $B_{g,n}$. To find this closure, at first,  consider  conjugations of the generators of $Q_{g,n}$ by $\sigma_{g+k}$, $k = 1, 2, \ldots, n-1$.

Conjugating generators from the $k$th column  of our table  by $\sigma_{g+k}$,  we have
$$
\sigma_{g+k}^{-1} a_{i,g+k} \sigma_{g+k} = a_{i,g+k} a_{i,g+k+1} a_{i,g+k}^{-1},~~i = 1, 2, \ldots, g.
$$

Conjugating generators from the $(k+1)$st column  of our table  by $\sigma_{g+k}$, we have
$$
\sigma_{g+k}^{-1} a_{i,g+k+1} \sigma_{g+k} =  a_{i,g+k},~~i = 1, 2, \ldots, g.
$$

Generators from all other columns commute with $\sigma_{g+k}$.

Hence,
for every generator $a_{ij}$ of $Q_{g,n}$ and every $\sigma_k$, $k = g+1, g+2, \ldots, g+n-1$, the element
$\sigma_k^{-1} a_{ij} \sigma_k$ lies in $Q_{g,n}$.

For the group $\widetilde{B}_n$ there exists the following exact sequence
$$
1 \longrightarrow \widetilde{P}_n \longrightarrow \widetilde{B}_n \longrightarrow S_n \longrightarrow 1.
$$
Let $m_{kl} = \sigma_{k-1}  \, \sigma_{k-2} \ldots \sigma_l$ for $l < k$ and $m_{kl} = 1$
in other cases. Then the set
$$
\Lambda_n = \left\{ \prod\limits_{k=g+2}^{g+n} m_{k,j_k}~ \vert ~1 \leq j_k
\leq k \right\}
$$
is a Schreier set of coset representatives of $\widetilde{P}_n$ in $\widetilde{B}_n$.

From the previous observations it follows that if $\alpha \in \Lambda$, then for every generator $a_{ij}$ of $Q_{g,n}$ the element
$\alpha^{-1} a_{ij} \alpha$ lies in $Q_{g,n}$.

Now we will consider the conjugations of the generators of $Q_{g,n}$ by the generators of $P_{g,n}$.

1) To prove the first decomposition, take  some element $h \in B_{g,n}$ and using the vertical decomposition, write it in the normal form:
$$
h = u_{g+1} u_{g+2} \ldots u_{g+n} \alpha,~\mbox{where}~u_k \in U_k,~~\alpha \in \Lambda_n.
$$
In every group $U_k$, $k = g+1, g+2, \ldots, g+n$, we have two subgroups:
$$
U_{g,k} = \langle a_{1k}, a_{2k}, \ldots, a_{gk} \rangle,~~\widetilde{U}_{k} = \langle a_{g+1,k}, a_{g+2,k}, \ldots, a_{k-1,k} \rangle \leq \widetilde{P}_n.
$$
We see that $U_{g,k}$ is a subgroup of $P_{g,n}$,  $\widetilde{U}_{k}$ is a subgroup of  $\widetilde{P}_{n}$ and $U_k = \langle U_{g,k}, \widetilde{U}_{k} \rangle$ is a free group. Define the projection $\pi_k : U_k \longrightarrow \widetilde{U}_k$ by the rules
$$
\pi_k(a_{ik}) = 1,~~i = 1, 2, \ldots, g;
$$
$$
\pi_k(a_{jk}) = a_{jk},~~j = g+1, g+2, \ldots, k-1.
$$
The kernel $\mathrm{Ker}(\pi_k)=\overline{U}_{g,k}$ is the normal closure $\langle U_{g,k} \rangle^{U_k}$ of $U_{g,k}$ into $U_k$. Hence,  element $u_k$ can be written  in the form
$u_k = \overline{u}_k \pi_k(u_k)$ for some $\overline{u}_k \in \overline{U}_{g,k}$. Denote for simplicity $w_{g+i} = \pi_{g+1}(u_{g+i})$, we can rewrite $h$ in the form
$$
h = (\overline{u}_{g+1} w_{g+1}) (\overline{u}_{g+2} w_{g+2}) \ldots (\overline{u}_{g+n} w_{g+n}) \alpha.
$$
Shifting all $w_{g+i}$ to the right we get
$$
h = \overline{u}_{g+1} \overline{u}_{g+2}^{w_{g+1}^{-1}}  \ldots \overline{u}_{g+n}^{w_{g+n-1}^{-1} w_{g+n-2}^{-1} \ldots w_{g+1}^{-1}} w_{g+1} w_{g+2} \ldots w_{g+n} \alpha.
$$

The image $\varphi_n(h) = w_{g+1} w_{g+2} \ldots w_{g+n} \alpha$. Hence, the element
$$
\overline{u}_{g+1} \, \overline{u}_{g+2}^{w_{g+1}^{-1}} \,  \ldots \, \overline{u}_{g+n}^{w_{g+n-1}^{-1} w_{g+n-2}^{-1} \ldots w_{g+1}^{-1}}
$$
lies in the kernel $\mathrm{Ker}(\varphi_n)=R_{g,n}$. We proved that any element from $R_{g,n}$ lies in the product
$$
\overline{U}_{g,g+n} \leftthreetimes (\overline{U}_{g,g+n-1} \leftthreetimes ( \ldots \leftthreetimes (\overline{U}_{g,g+2} \leftthreetimes \overline{U}_{g,g+1})\ldots )).
$$
On the other side, this product contains $Q_{g,n}$ and is normal in $B_{g,n}$. Hence, $R_{g,n}$ has the required decomposition.

2) Prove the second decomposition. Denote $G = \overline{V}_{g,1} \leftthreetimes (\overline{V}_{g,2} \leftthreetimes ( \ldots \leftthreetimes (\overline{V}_{g,g-1} \leftthreetimes \overline{V}_{g,g})\ldots ))$ the group from the right side of the decomposition of $R_{g,n}$.
Take an arbitrary element $h \in B_{g,n}$. By the theorem on the horizontal decomposition of $B_{g,n}$ it has the following normal form:
$$
h = v_1 v_2 \ldots v_g b, ~\mbox{where}~v_i \in V_{g,i}, b \in \widetilde{B}_n
$$
and $v_1 v_2 \ldots v_g$ lies in $R_{g,n}$.
Under the endomorphism $\varphi_n : B_{g,n} \longrightarrow \widetilde{B}_n$ element $h$ goes to the element $b$. Hence,
$G$ lies in $R_{g,n}$.

On the other side we must prove that $R_{g,n}$ lies in $G$. We know that $R_{g,n}$ is the normal closure of $Q_{g,n}$ in $B_{g,n}$. As was shown before, if $\alpha \in \Lambda_n$ is a coset representative of $\widetilde{P}_n$ into $\widetilde{B}_n$ and $a_{ij}$ is some generator of $Q_{g,n}$, then
$a_{ij}^{\alpha}$ lies in $Q_{g,n}$. Considering the conjugation of $a_{ij} \in Q_{g,n}$ by  generators of $P_{g,n}$ and using the Lemma \ref{l1} we get  an element which lies in $G$.
\end{proof}

In the case $g=1$ we have $P_{1,n} = P_{n+1}$ and  $R_{1,n} = \langle a_{12}, a_{13},$ $\ldots,$ $a_{1,n+1} \rangle \cong F_n$ and we get the decomposition which was found in \cite{L2}.

\begin{corollary}
$B_{1,n} \cong F_n \leftthreetimes B_n$.
\end{corollary}

\section{Some analog of the Hecke algebra for the braid group in the handlebody}

Let $q$ be some complex number. Recall that the Hecke algebra $H_{n}(q)$ is an associative $\mathbb{C}$-algebra with unit, which is generated by elements
$$
s_1, s_2, \ldots, s_{n-1},
$$
and is defined by the relations
$$
\begin{array}{ll}
s_i s_j = s_j s_i, & ~\mbox{for}~|i - j| > 1, \\
s_i s_{i+1} s_i = s_{i+1} s_i s_{i+1}, & ~\mbox{for}~1 \leq  i \leq n-2, \\
s_i^2 = (q-1) s_i + q, & ~\mbox{for}~i = 1, \ldots, n-1.
\end{array}
$$

The  algebra  $H_{n}(q)$ has the following linear basis:
$$
S = \{ (s_{i_1} s_{i_1-1} \ldots s_{i_1-k_1}) (s_{i_2} s_{i_2-1} \ldots s_{i_2-k_2}) \ldots (s_{i_p} s_{i_p-1} \ldots s_{i_p-k_p}) \}
$$
for $1 \leq i_1 < \ldots < i_p \leq n-1$. The basis $S$ is used in the construction of the Markov trace, leading to the HOMFLYPT or $2$-variable Jones polynomial (see \cite{J}).

The braid group in the solid torus $B_{1,n}$ is the Artin group of the Coxeter group of type $\mathcal{B}$, which is related to the Hecke algebra of type $\mathcal{B}$.
The generalized Hecke algebra of type $\mathcal{B}$, $H_{1,n}(q)$ is defined by S. Lambropoulou in \cite{L}. $H_{1,n}(q)$ is isomorphic to the affine Hecke algebra of type $\mathcal{A}$, $\widetilde{H}_{n}(q)$. A unique Markov trace is constructed on the algebras $H_{1,n}(q)$ that leads to an invariant for links in the solid torus, the universal analogue of the HOMFLYPT polynomial for the solid torus.

The algebra $H_{1,n}(q)$ is generated by elements
$$
t, t^{-1}, s_1, s_2, \ldots, s_{n-1},
$$
and is defined by the relations
$$
\begin{array}{ll}
s_i s_j = s_j s_i, & ~\mbox{for}~|i - j| > 1, \\
s_i s_{i+1} s_i = s_{i+1} s_i s_{i+1}, & ~\mbox{for}~1 \leq  i \leq n-2, \\
s_i^2 = (q-1) s_i + q, & ~\mbox{for}~i = 1, 2, \ldots, n-1,\\
s_1 t s_1 t  = t s_1 t s_1, &  \\
t s_i = s_i t, & ~\mbox{for}~i > 1.
\end{array}
$$
Hence,
$$
H_{1,n}(q) = \frac{\mathbb{C} [B_{1,n}]}{\langle \sigma_i^2 - (q-1) \sigma_i - q \rangle}.
$$
Note that in $H_{1,n}(q)$ the generator $t$ satisfies no polynomial relation, making the algebra $H_{1,n}(q)$ infinite dimensional.
If we set $t=0$ in $H_{1,n}(q)$, we obtain the Hecke algebra $H_n(q)$.

In $H_{1,n}(q)$ are defined in \cite{L2} the elements
$$
t_i = s_i s_{i-1} \ldots s_1 t s_1 \ldots s_{i-1} s_i,~~~t'_i = s_i s_{i-1} \ldots s_1 t s_1^{-1} \ldots s_{i-1}^{-1} s_i^{-1}.
$$
It was then proved that the following sets form linear bases for $H_{1,n}(q)$:
$$
\Sigma_n = t^{k_1}_{i_1} t^{k_2}_{i_2} \ldots t^{k_r}_{i_r} \sigma, ~\mbox{where}~1 \leq i_1 < \ldots < i_r \leq n-1,
$$
$$
\Sigma'_n = (t_{i_1}')^{k_1} (t_{i_2}')^{k_2} \ldots (t_{i_r}')^{k_r} \sigma, ~\mbox{where}~1 \leq i_1 < \ldots < i_r \leq n-1,
$$
where $k_1, k_2, \ldots, k_r \in \mathbb{Z}$ and $\sigma$ a basis element in $H_{n}(q)$. The basis $\Sigma'_n$ is used in \cite{L, L2} for constructing a Markov trace on $\bigcup_{n=1}^{\infty} H_{1,n}(q)$ and a universal HOMFLYPT-type invariant for oriented links in the solid torus.

\begin{definition} Let $q \in \mathbb{C}$. The algebra $H_{g,n}(q)$ is an associative algebra over $\mathbb{C}$ with unit that is generated by
$$
t_1^{\pm 1}, t_2^{\pm 1}, \ldots, t_g^{\pm 1}, s_{g+1}, s_{g+2}, \ldots s_{g+n-1}
$$
and is defined by the following relations
$$
\begin{array}{ll}
s_i s_j = s_j s_i, & ~\mbox{for}~|i - j| >1, \\
s_i s_{i+1} s_i = s_{i+1} s_i s_{i+1}, & ~\mbox{for}~i = g+1, g+2, \ldots, g+n-2, \\
s_i^2 = (q-1) s_{i} + q, & ~\mbox{for}~i = g+1, g+2, \ldots, g+n-1, \\
t_k s_i = s_i t_k, & ~\mbox{for}~k \geq 1, i \geq g+2, \\
t_k (s_{g+1} t_k s_{g+1}) = (s_{g+1} t_k s_{g+1}) t_{k}, & ~\mbox{for}~k = 1, 2, \ldots, g, \\
t_k (s_{g+1} t_{k+l} s_{g+1}) = (s_{g+1} t_{k+l} s_{g+1}) t_{k}, & ~\mbox{for}~k = 1, 2, \ldots, g-1; l = 1, 2, \ldots, g-k.
\end{array}
$$
\end{definition}

\begin{remark}
More natural to consider the generators $s_{1}, s_{2}, \ldots s_{n-1}$ instead $s_{g+1}, s_{g+2}, \ldots s_{g+n-1}$, but for technical reasons we will use our notation.
\end{remark}

We see that
$$
H_{g,n}(q) = \frac{\mathbb{C} [B_{g,n}]}{\langle \sigma_i^2 - (q-1) \sigma_i - q \rangle}.
$$

If we consider the algebra $H_n(q)$ as a vector space over $\mathbb{C}$, then it is isomorphic to the vector space $\mathbb{C}[S_n]$. Thus to study $H_{g,n}(q)$, we define a group $G_{g,n}$ as the quotient of $B_{g,n}$ by the relations
$$
\sigma_i^2 = 1,~~~i = 1, 2, \ldots, n-1.
$$
Denote the natural homomorphism $B_{g,n} \longrightarrow G_{g,n}$ by $\psi$ and let $\psi(a_{ij}) = b_{ij}$.

\begin{theorem}\label{t3}
The group $G_{g,n}$ is the semi-direct product $G_{g,n} = F_g^{n} \leftthreetimes S_n$ of the direct product of $n$ copies of free group $F_g$ of rank $g$ and the symmetric group $S_n$.
\end{theorem}

\begin{proof}
We know that $B_{g,n} = R_{g,n} \leftthreetimes \widetilde{B}_n$ and
$$
R_{g,n} = \overline{U}_{g,g+n} \leftthreetimes (\overline{U}_{g,g+n-1} \leftthreetimes ( \ldots \leftthreetimes (\overline{U}_{g,g+2} \leftthreetimes \overline{U}_{g,g+1})\ldots )).
$$
On the other side, $U_k$, $k= g+1, g+2, \ldots, g+n$, is generated by two subgroups:
$$
U_{g,k} = \langle a_{1k}, a_{2k}, \ldots, a_{gk} \rangle ~\mbox{and}~\widetilde{U}_{k} = \langle a_{g+1,k}, a_{g+2,k}, \ldots, a_{k-1,k} \rangle \leq \widetilde{P}_n.
$$
Note that under the homomorphism $\psi$ the subgroups $\widetilde{U}_{k}$ go to 1. Hence, $\psi(\overline{U}_{g,g+k}) = \psi(U_{g,g+k}) \cong F_g$ and
$$
\psi(R_{g,n}) = \psi(U_{g,g+n}) \leftthreetimes (\psi(U_{g,g+n-1}) \leftthreetimes ( \ldots \leftthreetimes (\psi(U_{g,g+2}) \leftthreetimes \psi(U_{g,g+1}))\ldots ))
$$
is the semi-direct product of  $n$ copies of free group $F_g$. Also, $G_{g,n} = \psi(R_{g,n}) \leftthreetimes \psi(\widetilde{B}_n) \cong \psi(R_{g,n}) \leftthreetimes S_n$.

Let us show that in fact $\psi(R_{g,n})$ is the direct product of $n$ copies of free group $F_g$. To do it it is enough to prove that
$$
\psi(R_{g,n}) = \psi(U_{g,g+n}) \times \psi(R_{g,n-1}).
$$
Denote $b_{ij} = \psi(a_{ij})$, consider the relations (\ref{co1})-(\ref{co4}) and find their images under the homomorphism $\psi$. The relation (\ref{co1}) goes to the relation
$$
b_{i,g+j}^{-\varepsilon} b_{g+j,g+n} b_{i,g+j}^{\varepsilon} = (b_{i,g+n} b_{g+j,g+n})^{\varepsilon} b_{g+j,g+n} (b_{i,g+n} b_{g+j,g+n})^{-\varepsilon},
$$
but $b_{g+j,g+n} = 1$ and we have the trivial relations.

The relation (\ref{co2}) goes to the relation
$$
b_{l,g+j}^{-\varepsilon} b_{l,g+n} b_{l,g+j}^{\varepsilon} = (b_{l,g+n} b_{g+j,g+n})^{\varepsilon} b_{l,g+n} (b_{l,g+n} b_{g+j,g+n})^{-\varepsilon},
$$
but $b_{g+j,g+n} = 1$ and we have the relation
$$
b_{l,g+j}^{-\varepsilon} b_{l,g+n} b_{l,g+j}^{\varepsilon} =  b_{l,g+n}.
$$

The relation (\ref{co3}) goes to the relation
$$
b_{i,g+j}^{-\varepsilon} b_{l,g+n} b_{i,g+j}^{\varepsilon} = [b_{i,g+n}^{-\varepsilon}, b_{g+j,g+n}^{-\varepsilon}]^{\varepsilon} b_{l,g+n} [b_{i,g+n}^{-\varepsilon}, b_{g+j,g+n}^{-\varepsilon}]^{-\varepsilon},~~~i < l < g+j,
$$
but $b_{g+j,g+n} = 1$ and we have the relation
$$
b_{i,g+j}^{-\varepsilon} b_{l,g+n} b_{i,g+j}^{\varepsilon} =  b_{l,g+n},~~~i < l < g+j.
$$

The relation (\ref{co4}) goes to the relation
$$
b_{i,g+j}^{-\varepsilon} b_{l,g+n} b_{i,g+j}^{\varepsilon} =  b_{l,g+n},~~~l < i < g+j < g+n.
$$

The image of $\widetilde{B}_n \leq B_{g,n}$ is isomorphic to $S_n$. Thus from the decomposition $B_{g,n} = R_{g,n} \leftthreetimes \widetilde{B}_n$ we get the required  decomposition for $G_{g,n}$.
\end{proof}

From this theorem we have

\begin{corollary}
Every element $h \in G_{g,n}$ has the unique normal form
$$
h = h_1 h_2 \ldots h_n \alpha,
$$
where $h_i$ is a reduced word in the free group
$$
\psi(U_{g,g+i}) = \langle b_{1, g+i}, b_{2, g+i}, \ldots, b_{g, g+i} \rangle \cong F_g,~~~i = 1, 2, \ldots, n,
$$
$\alpha \in \Lambda_n$ is a coset representative of $\psi(R_{g,n})$ in $G_{g,n}$.

\end{corollary}

Using this normal form and the fact that these normal forms form a linear basis of $\mathbb{C}[G_{g,n}]$, we can try to define a basis of $H_{g,n}(q)$.
In the algebra $H_{g,n}(q)$ define the following elements
$$
\begin{array}{cccc}
t_{1,g+1} = t_1,~~ & t_{1,g+2}, & \ldots, & t_{1,g+n}, \\
t_{2,g+1} = t_2,~~ & t_{2,g+2}, & \ldots, & t_{2,g+n}, \\
\ldots ~~& \ldots & \ldots & \ldots \\
t_{g,g+1} = t_g, ~~& t_{g,g+2}, & \ldots, & t_{g,g+n},
\end{array}
$$
where
$$
t_{ij} = s_{j-1} s_{j-2} \ldots s_1 t_{i,g+1} s_1 \ldots s_{j-2} s_{j-1}, ~~1 \leq i \leq g,~~g+2 \leq j \leq g+n.
$$
These elements correspond to the elements $b_{ij}$ from $G_{g,n}$, which were defined in \ref{t3} as the images of the elements $a_{ij}$ under the map $\psi$.  It is not difficult to see that the elements
$$
t_{1, g+i}, t_{2, g+i}, \ldots, t_{g, g+i}
$$
generate a free group of rank $g$.

The algebra $H_{g,n}(q)$ contains the subalgebra with the set of generators $s_{g+1}, s_{g+2}, \ldots s_{g+n-1}$, which is isomorphic to $H_n(q)$.
Let $\Sigma_{g,n}$ be the following set in $H_{g,n}(q)$:
$$
u_1 u_2 \ldots u_n \sigma,
$$
where $u_i$ is a reduced word in the free group
$$
\langle t_{1, g+i}, t_{2, g+i}, \ldots, t_{g, g+i} \rangle \cong F_g,~~~i = 1, 2, \ldots, n,
$$
$\sigma$ is a basis element in $H_{n}(q)$.

\begin{conjecture}
There is an isomorphism $H_{g,n}(q) \cong \mathbb{C}[G_{g,n}]$ as $\mathbb{C}$-modules. In particular, the set $\Sigma_{g,n}$ is a basis of the algebra  $H_{g,n}(q)$.
\end{conjecture}

The algebra $H_{2,n}(q)$ is the subject of study in \cite{KL}, where one set of
elements in $H_{2,n}(q)$, different from our $\Sigma_{2,n}$, is proved to be a
spanning set for the algebra.

\medskip

At the end, we formulate the following question for further investigation.

\begin{question}
Is it possible to define a Markov trace on the algebra $\bigcup_{n=1}^{\infty} H_{g,n}(q)$ and construct some analogue  of the HOMLYPT polynomial that is an invariant of  links in the handlebody $\mathcal{H}_g$?
\end{question}


\begin{thebibliography}{HD}

\bibitem{Bir}
J. S. Birman,
\textit{Braids, Links, and Mapping Class Groups},
\emph{Annals of Math. Studies} 82, Princeton University Press, 1974.


\bibitem{DiLaPr}   I. Diamantis, S. Lambropoulou, J. Przytycki, \textit{Topological steps
toward the HOMFLYPT skein module of the lens spaces $L(p,1)$ via braids}, \emph{J.
Knot Theory Ramifications}, \textbf{25}, N~ 14 (2016), 1650084 (26 pages). See
arXiv:1604.06163.

\bibitem{J} V. F. R. Jones, \textit{Hecke algebra representations of braid groups and links polynomials}, Ann. Math., 1984, \textbf{126}, 335--388.

\bibitem{HL}   R. Haring-Oldenburg, S. Lambropoulou \textit{Knot theory in handlebodies},
\emph{J. Knot Theory Ramifications}, \textbf{11}, N~6 (2002), 921--943.

\bibitem {KL}
D. Kodokostas and S. Lambropoulou, \textit{A spanning set and potential basis of the mixed  Hecke algebra on two fixed strands}, arXiv: 1704.03676.

\bibitem {L} S. Lambropoulou, \textit{Solid torus links and Hecke algebras of  $B$-type}, Proc. of the Conf. of Quantum Topology, D. N. Yetter ed., World Scientific Press, 1994.

\bibitem {L1} S. Lambropoulou, \textit{Braid structure in knot complements, handlebodies and 3-manifold}, Knots in Hellas '98 (Delphi), Proc. of International Conference Knot Theory and its Ramifications, WS, 274--289.

\bibitem {L2} S. Lambropoulou, \textit{Knot theory related to generalized and cyclotomic Hecke algebras of type B}, \emph{J. Knot Theory Ramifications}, \textbf{8}, N~5 (1999), 621--658.

\bibitem {LR1}    S. Lambropoulou, C.P. Rourke \textit{Markov's theorem in 3-manifolds}, \emph{Topology and its Applications}, 1997,  \textbf{78}, 95--122.

\bibitem {LR2}   S. Lambropoulou, C.P. Rourke \textit{Algebraic Markov equivalence for links
in 3-manifolds}, \emph{Compositio Mathematica}, 2006, \textbf{142}, 1039--1062.

\bibitem{M}
A.~A. Markov,
\textit{Foundations of the algebraic theory of braids},
\emph{Trudy Mat. Inst. Steklov [Proc. Steklov Inst. Math.]}, 1945,
\textbf{16}, 1--54.


\bibitem {S} A. B. Sossinsky, \textit{Preparation theorem for isotopy invariants of links in 3-manifold}, Quantum Groups. Proc. of the Conf. on Quantum Groups. Berlin a.: Springer-Verl., 1992, 354--362 (Lecture Notes in Math.; N 1510).



\bibitem {V} V. V. Vershinin, \textit{On braid groups in handlebodies}, \emph{Siberian Math. J}., \textbf{39}, N~4 (1998), 645--654. Translation from Sibirsk. Mat. Zh., \textbf{39}, N 4 (1998), 755--764.

\bibitem {V1} V. V. Vershinin, \textit{Generalization of  braids from a homological point of view}, \emph{Sib. Adv. Math}., \textbf{9}, N~2 (1999), 109--139.


\end{thebibliography}
\end{document}